\documentclass[a4paper]{article} 
\usepackage{amssymb,amsmath}
\usepackage{graphicx}
\newcommand\eps{\varepsilon}

\usepackage{fullpage}
\newtheorem {conj}{Conjecture}[section]
\newtheorem {puzzle}{Puzzle}[section]

\newtheorem {thm}{Theorem}[section]
\newtheorem {lem}[thm]{Lemma}
\newtheorem {pr}[thm]{Proposition}

\newtheorem {df}[thm]{Definition}

\newcommand{\qed}{\hfill \ensuremath{\Box}}
\newenvironment{proof}[1][Proof.]{\begin{trivlist}
\item[\hskip \labelsep {\bfseries #1}]}{\qed\end{trivlist}}

\def\N{\mathbb N}
\def\p{p}

\title{Enumeration and asymptotics of restricted compositions having the same number of parts}

\author{Cyril Banderier\thanks{Laboratoire d'Informatique de Paris
 Nord, UMR CNRS 7030, Institut Galil\'ee, Universit\'e Paris 13, 99
avenue Jean-Baptiste Cl\'ement, 93430 Villetaneuse, France.
\texttt{http://www-lipn.univ-paris13.fr/$\sim$banderier/}}
\/ and Pawe{\l} Hitczenko\thanks{Departments of Mathematics and Computer
 Science, Drexel University, Philadelphia, PA 19104, USA. 
\texttt{http://www.math.drexel.edu/$\sim$phitczen/}
}
\thanks{supported in part by the NSA grant \#H98230-09-1-0062}
\thanks{corresponding author. Article submitted in April 2010, minor
  revision in October 2011. To appear in Discrete Applied Mathematics.}} 
\date{\today \\\vspace{1mm} Dedicated to the memory of Philippe Flajolet.}

\begin{document}

\maketitle
\noindent{\bf Abstract:} We study pairs and~$m$--tuples of
compositions of a positive integer~$n$ with parts restricted to a
subset~$\mathcal P$ of positive integers. 
We obtain some exact enumeration results for the number of tuples of such compositions
having the same number of parts. Under the uniform probability model, we obtain the asymptotics 
for the probability that two or, more generally,~$m$ randomly and independently chosen compositions of~$n$ have the same number of parts. 
For a large class of compositions, we show how a nice interplay
between complex analysis and probability theory allows to get full
asymptotics for this probability.
Our results extend an earlier work of B\'ona and Knopfmacher. 
While we restrict our attention to compositions, our approach is also
of interest for tuples of other 
 combinatorial structures having the same number of parts. 
\vspace{4mm}

\noindent{{\bf Keywords:} integer composition, pairs of combinatorial structures, local limit theorem, asymptotics of D-finite sequences, diagonal of algebraic generating function.
\vspace{-3mm}

\section{Introduction} In this note, we study tuples of compositions
of positive integers having the same number of parts,
 and  the 
 asymptotics of related generating functions satisfying some differential equations. Let us recall that a composition of a positive integer~$n$ is any~$k$--tuple~$(\kappa_1,\dots,\kappa_k)$,~$k\ge1$, of positive integers that sum up to~$n$. The~$\kappa_j$'s are called the parts (or summands) of a composition. It is elementary and well--known (see, e.g.~\cite{and}) that there are~${n-1\choose k-1}$ compositions of~$n$ with~$k$ parts, and thus there are~$2^{n-1}$ compositions of~$n$. By restricted compositions we mean compositions whose parts are confined to be in a fixed subset~$\mathcal P$ of~$\N$. 

The main motivation for this work is a recent paper~\cite{bk} in
which the authors studied pairs of compositions with the same number
of parts. Our extension of this work is directly connected  to the question of
obtaining the asymptotics of coefficients of functions satisfying a linear
 differential equation which, despite the deep work by Fabry, Frobenius,
 Fuchs, Picard and other analysts more than one century ago, remains
 open and is conjectured to be undecidable. We present here a new way
 to use probability theory {\em in addition to} complex analysis in order to solve this problem
for a large class of functions. 

In their paper~\cite{bk}, B\'ona and Knopfmacher studied  the asymptotic probability that  two randomly and independently chosen
compositions of~$n$ have the same number of parts. Furthermore, relying on the generating
function approach, for a few {\it specific} subsets~$\mathcal P$ they
addressed the same question for pairs of restricted compositions. In
each of these cases this probability is asymptotic to~$C/\sqrt n$
with~$C$ depending on~$\mathcal P$. Our main aim here is to extend
these results. First, we show that this asymptotics is
universal. That is, we show that for an {\it arbitrary} subset~$\mathcal P$ 
containing two relatively prime elements the probability that two independently chosen random compositions of~$n$ with parts restricted to~$\mathcal P$ 
have the same number of parts is asymptotic to~$C/\sqrt n$. The value of~$C$ depends, generally, on~$\mathcal P$ and is explicit.
 (See our Theorem~\ref{thm:samenoparts} and subsequent remarks, which include e.g. a correction of a constant appearing in~\cite{bk}.) 
Secondly, we consider the same question for~$m>2$ and we show that in this case the sought probability is asymptotic to~$C/\sqrt{n^{m-1}}$ for an explicitly given constant~$C$ whose value depends on~$\mathcal P$ and~$m$ only.
(See our Theorem~\ref{thm:msamenoparts}.)

B\'ona and Knopfmacher's approach  relied on complex analysis; 
the universality of using a more probabilistic technique was then noticed by B\'ona and Flajolet~\cite{bf}, where certain types of random trees were studied. 
Our approach is in one sense a mixture of complex analysis (which gives the full asymptotics expansion, up to a multiplicative constant, and with the price of heavy computations),
and probability theory (a local limit theorem which gives without any heavy computation the first asymptotic term, and therefore gives access to the multiplicative constant, but intrinsically  no access to further asymptotic terms).
B\'ona and Flajolet obtained, in particular, a general statement indicating how local limit theorem can help in evaluating probabilities that two independently chosen random structures of the same size have the same number of components 
(this is their Lemma~6 in~\cite{bf}, which corresponds to our Lemma~\ref{lem:llc0} for Gaussian density with a slightly different proof. Our Lemma~\ref{lem:llc0} was obtained independently, but later). 
 As we will see, these statements remain true if one considers more than two random structures.

In Section~2, we present our model. We proceed in Section~3 with some examples (and {\em en passant}, some nice questions in computer algebra) and argue on the intrinsic limitations of an approach relying only on complex analysis.
This serves~as a motivation for introducing the local 
 limit law result in Section~4, which finds application in Section~5, thus solving the initial problem of the asymptotic evaluation  of the probability that tuples of compositions have the same number of parts. We conclude with some perspectives in Section~6.

\section{Generating functions for pairs of compositions having the same number of parts}

Let us consider compositions with parts in a set~$\mathcal P$ 
(a fixed subset of~$\N$). To avoid trivial complications caused by the fact that there may be no compositions of a given~$n$ with all parts from~$\mathcal P$, 
we assume that~${\mathcal P}$ has at least two elements that are
relatively prime 
(except when explicitly stated otherwise). 

We introduce the generating function of the parts 
$\p(z) = \sum_{j\in {\mathcal P}} \p_j z^j$,
($\p_j$ is not necessarily 0 or 1, it can then be seen as the possible
colors or the weight of part~$j$). We thus assume that the~$p_j$'s
are non--negative real numbers such that~$\sum_{j\in \mathcal
 P}p_j>1$. This last condition is to ensure supercriticality of our
scheme (see Section~\ref{sec:local} below for more details). In the
classical situation when~$p_j$ is 0 or~1, this condition holds
automatically.
Denote by 
\begin{equation} \label{compsch}
P(z,u)= \sum_{n\geq 0,k\geq 0} P_{n,k}u^k z^n =\frac{1}{1-u \p(z)}
\end{equation}
the bivariate generating function 
of compositions of~$n$ where~$k$ encodes its number of
parts, and where the ``size'' of the composition is~$n$. 

With a slight abuse of notation, the corresponding univariate generating function is
\begin{equation}
P(z)= \sum_{n\geq 0} P_n z^n = \frac{1}{1-\p(z)} \,.
\end{equation}

This terminology is classical. For example, here are all the
compositions of 5 with 3 parts from the set~$\mathcal{P}=\{1,2,3,4,10\}$:
$5=1+1+3=1+3+1=3+1+1=1+2+2=2+1+2=2+2+1$. Accordingly,~$P_{5,3}=6$.

Let~$X_n^{\mathcal P}$ 
be the random variable giving the number of parts 
in a random composition of~$n$ with parts belonging to~$\mathcal
P$. 
Random means that we consider the uniform
distribution among all compositions of~$n$ with parts belonging to~$\mathcal P$.

Given two subsets~${\mathcal P}_1$ and~${\mathcal P}_2$ of~$\N$, we
consider 
 the probability~$\pi_n:=\operatorname{Pr}(X_n^{{\mathcal P}_1}=X_n^{{\mathcal P}_2})$
that a random composition of~$n$ with
parts in~${\mathcal P}_1$ has the same number of parts as a random
composition with parts in~${\mathcal P}_2$. 
We assume throughout that, whenever two 
such compositions are chosen, they are chosen independently and from now on we will not be explicitly mentioning it.
We then introduce the generating function~$D(z)$ of the number of pairs of compositions (the
first one with parts in~$\mathcal P_1$, the second one with parts in
$\mathcal P_2$) having the same size and the same number of parts. ($D$ stands for
``double'' or ``diagonal'', as~$D(z)$ can be obtained as a diagonal of
multivariate function.)

That is, we consider all~$k$-tuples of elements of~$\mathcal P_1$
and all~$k$-tuples of elements of~$\mathcal P_2$ such that their sum is~$n$.
For a fixed~$n$, let~$D_n$ be the total number of such configurations
(i.e., we sum over all~$k$).

In the next section, we deal with some interesting examples for which
we get explicit formulas.

\section{Some closed-form formulas}

\subsection{An example on tuples of domino tilings}\label{sec:dom}

Consider the 
classical combinatorial problem of  
tiling a~$2\times n$ strip by dominoes.
Any tiling is thus a sequence of either one horizontal domino or 
2 vertical dominoes.
The generating function of domino tilings 
is thus~$P(z)=\operatorname{Seq}(z+z^2) = \frac{1}{1-z-z^2}$,
which is the generating function of~$P_n=F_{n+1}$, where~$F_n$ is the Fibonacci number~$F_n=F_{n-1}+F_{n-2}$,~$F_0=0$,~$F_1=1$.
(Equivalently, the Fibonacci recurrence reflects the fact
that removing a horizontal domino on the top of an existing~$2 \times
n$ tiling leads to a~$2 \times (n-1)~$ tiling, while removing 
2 vertical dominoes on the top leads to a~$2 \times (n-2)~$ tiling.)
Let us now consider  a less trivial  question, which is archetypal
of the problem we consider in this article 
(note that it has a closed-form solution but we will address later in
this article similar problems having no such nice closed-form
solution):

\begin{puzzle}  Each of~$m$ children makes a tiling of a~$2\times n$ strip.
What is the probability~$\pi_n$ that these~$m$ tilings all have the same
number of vertical dominoes, when~$n$ gets large?
\end{puzzle}

For~$m=2$, the number of pairs is given by~$D(z)=[t^0] \sum_{k\geq 0}
p^k(z t) p^k(1/t)$, where~$\p(z)=z+z^2$, and the Cauchy formula gives
\begin{equation}D(z)=\frac{1}{2i\pi}\oint \frac{1}{1-\p(z t) \p(1/t)} \frac{dt}{t}= 
\frac{1}{2i\pi}\oint
\frac{{Num}(z,t)}{{Den}(z,t)} dt \,,\end{equation}
where~${Num}$ and~${Den}$ are polynomials in~$z,t$.
Let~$Z(z)$ be any root of~${Den}$, i.e.~${Den}(z,Z)=0$, such that~$Z$ is
inside the contour of integration for~$z\sim0$. Then, a residue computation gives: 
$$D(z)=\sum_Z
\frac{{Num}(z,Z)}{\partial_{t}{Den}(z,Z)}=\frac{1}{\sqrt{z^4-2z^3-z^2-2z+1}}~$$
$$= 1+z+2z^2+5z^3+11z^4+26z^5+ 63z^6+153z^7+ 376z^8+ 931z^9+O(z^{10})\,.$$
This is the sequence A051286 from~\cite{eis} $D_n=\sum_{k=0}^n \binom{n-k}{k}^2$, B\'ona and Knopfmacher~\cite{bk} gives a bijective proof that
 it is also the Whitney number of level~$n$ of the lattice of the ideals of the fence of order~$2 n$. 
The probability that 2 tilings of a~$2\times n$ strip
have the same number of vertical dominoes is therefore (via
singularity analysis, which can be done automatically with some
computer algebra systems, e.g. via  the
{\em equivalent} command of Bruno Salvy, in the Algolib Maple package available
at {\tt http://algo.inria.fr/libraries}):
\begin{equation}\label{eqn:pn}\pi_n=D_n/P_n^2\sim \frac{5^{3/4}}{2\sqrt{\pi} \sqrt{n}}+
\frac{5^{1/4} (\frac{11}{32}- \frac{\sqrt 5}{4})}{\sqrt{\pi} n^{3/2}}
 +O(\frac{1}{\sqrt{n^5}}) 
\approx \frac{.9432407854}{\sqrt n} \,.
\end{equation}
Note that this is consistent with the constant~$C$ given in Equation~(2.10)
in~\cite{bk}. Our computations are available on-line in a Maple session\footnote{See  {\tt http://www-lipn.univ-paris13.fr/$\sim$banderier/Pawel/Maple/}.}. 
Note that as Maple does not always simplify algebraic numbers
like humans would do (some denesting options are missing), we used here
some of our own denesting recipes 
 so that these nested radicals become more readable for human eyes.

For~$m=3$, it is possible to compute the diagonal~$D(z)$ via creative 
telescoping (as automated in Maple via the {\em MGfun} package
of Fr\'ed\'eric Chyzak or in Mathematica via the package {\em HolonomicFunctions} of Christoph Koutschan). This leads to the following differential equation:
$$0=\left( 4z^{7}+7z^{6}+7z^{5}+15z^{4}+41z^{2}+z+1 \right)D \left( z \right)~$$
$$+ \left( 5z^{8}+12z^{7}+7z^{6}+62z^{5}+88z^{3}+z^{2}+6z-1 \right) {\frac {d}{dz}}D \left( z \right)~$$
$$+z \left( z^{2}+1 \right) \left( z^{4}-z^{3}+5z^{2}+z+1 \right) \left( z^{2}+4z-1 \right) {\frac {d^{2}}{dz^{2}}}D \left( z \right) .$$
Here, the so-called Frobenius method gives the basis of the vector space of solutions of this ODE, under the form of local formal solutions around any singularity, by using the associated indicial polynomial (see~\cite[Chapter VII.9]{fs}).
In full generality, the dominating singularity of $D(z)$ is $z=\rho^m$, where $\rho$ is the radius of convergence of $1/(1-P(z))$; this can be proven via our Theorem~\ref{thm:msamenoparts}.
In our case, the Frobenius method gives around the dominating singularity $\zeta:=\sqrt5-2$:

\begin{eqnarray*}D(z)&=& \lambda_1    \left(  \frac{80 +41\sqrt5}{90}  \ln(z-\zeta)  + \frac{8+5\sqrt5}{9} \ln(z-\zeta) + O((z-\zeta)^2)  \right)\\
&+& \lambda_2  \left(  1-\frac{8+5\sqrt5}{9}(z-\zeta)+\frac{207+89 \sqrt5}{81}  (z-\zeta)^2+O((z-\zeta)^3)  \right)\\
\end{eqnarray*}
for some unknown coefficients $\lambda_1, \lambda_2$ (related to the so-called Stokes constants or connection constants)\footnote{\label{foot:frob}Note that, as typical with
 the Frobenius method (or also with the
Birkhoff-Tritjinski method, see~\cite{WimpZeilberger85}), it is not always possible to decide the
connection constant(s); in the next sections, we give a rigorous
probabilistic approach which allows to get this constant, and therefore
full asymptotics by coupling it with the Frobenius method!}.
However, only the first summand 
 contributes to the asymptotics of $D_n$
and a numerical scheme of our own allows to determine (with the help of the heuristic LLL algorithm) the value of $\lambda_1$.
Using singularity analysis then leads to 
$$\pi_n\sim 
\frac{5 \sqrt{15} }{6} \frac {1}{\pi n}+
\frac{5 (10 \sqrt3-9\sqrt {15})}{54} \frac{1}{\pi n^2}+O(1/n^3)
\approx \frac{1.027340740}{n} \,.$$

This asymptotics also proves that~$D(z)$ is not an algebraic function
(the local basis of the differential equation involves a logarithmic term).

\noindent For~$m=4$, creative telescoping leads to the following
differential equation:
{\tiny
$$ 2 \left(132{z}^{16}-3563{z}^{15}+\dots+110 \right) D \left( z\right) +2 \left( 209474{z}^{14}+\dots -1581z \right) {\frac {d}{dz}}D \left( z \right)~$$
$$ + \left( 704{z}^{18}+\dots -10143{z}^{2} \right) {\frac {d^{2}}{d{z}^{2}}}D(z) +z \left( 165{z}^{18}+\dots-55 \right) {\frac {d^{3}}{d{z}^{3}}}D \left( z \right)$$
$$ +{z}^{2} \left( z-1 \right) \left( z+1 \right) \left( {z}^{2}+z+1 \right) \left( {z}^{2}-7z+1 \right) \left( {z}^{2}-z+1 \right) \left( {z}^{4}+\dots+1 \right) 
 \left( 11{z}^{6}+\dots +11 \right) {\frac {d^{4}}{d{z}^{4}}}D \left( z \right) \,.$$}
Using the Frobenius method 
and a numerical scheme of ours, this leads to
$$\pi_n\sim \frac {25}{8} \frac{5^{1/4}\sqrt 2}{\sqrt{\pi n^3}}
+\frac{5}{256} \frac{5^{1/4} \sqrt{2} (47 \sqrt{5}- 240)} 
{\sqrt{\pi^3 n^5}} +O(\frac{1}{\sqrt{n^7}})
\approx \frac{1.186814138}{\sqrt{n^3}}\,.$$
It is noteworthy that this asymptotics is compatible with the fact that~$D(z)$ could be an
algebraic function. However, a guess based on Pad\'e approximants 
fails to find any algebraic equation. What is more,
the index of nilpotence~$\mod 2, 3, 5, 7, 11$ of~$D(z)$ is 3 (i.e. the
smallest~$i$ such that~$(d/dz)^i = L \mod p$ is~$i=3$ for primes~$p=2, 3,
5, 7, 11$... and~$L$ is the above irreducible unreadable linear differential operator
cancelling~$D(z)$). Therefore, according to a conjecture of
Grothendieck on the~$p$-curvature (see~\cite{BostanSchost}), the
function is not algebraic.

For~$m=5$, $D(z)$ is a non algebraic function satisfying a
differential equation of order 6 and of degree 38, which leads to 
$$\pi_n\sim \frac{25 \sqrt 5}{4} \frac{1}{\pi^2 n^2} \approx \frac{1.416006588}{n^2}\,.$$

The closed form of the coefficients is~$D_n(m)=\sum_{k=0}^n \binom{n-k}{k}^m$, as can also be obtained via a bijective proof approach.
It is possible to get their asymptotics via the Laplace method or our Theorem~\ref{thm:msamenoparts},
this leads to $\pi_n \sim C_m / \sqrt{ (\pi n)^{m-1} }$ with $C_m=(5^{3/4})^{m-1}/\sqrt{2^{m-1} m}$.
This allows us to give a proof of the following claim (which was a conjecture by Paul D. Hanna, see \cite[A181545]{eis}): 
\begin{pr} 
$D_{n+1}(m)/D_n(m) \sim  (F_m \sqrt{5} + L_m )/2$,
where $L_m$ are the Lucas numbers, defined by the same recurrence as the Fibonacci numbers $F_m$, but with different initial conditions, namely $L_0=2$ and $L_1=1$.
\end{pr}
\begin{proof}
$$\frac{D_{n+1}(m)}{D_n(m)}=  \frac{\pi_{n+1}(m) P_{n+1}^m }{\pi_{n}(m) P_{n}^m } 
= \frac{C_m/ \sqrt{(\pi (n+1))^{m-1}}  P_{n+1}^m }{C_m/ \sqrt{(\pi n)^{m-1} } P_n^m}$$
$$=  \left(\frac{n}{n-1}\right)^{\frac{m-1}{2}}  \left(\frac{P_{n+1}}{P_n}\right)^m$$
$$\sim  \left(1+ (1/2-m/2) \frac{1}{n}+O(\frac{1}{n^2})\right)    \left(\frac{1}{\rho} \,  (1+O(\eps^n))\right)^m \sim \frac{1}{\rho^m}$$
where $\rho=p^{-1}(1)$ and the asymptotics for $P_n$ is explained in detail in the next section (Equation~\ref{asymptPn}).
In the case of $p(z)=z+z^2$, the claim then follows from $\rho=1/\phi$ and 
the exact formula $\phi^m= (F_m \sqrt{5} + L_m )/2$.
\end{proof}

Note that for all odd values of $m>2$, the presence of an integer power of~$\pi$ in the asymptotics of $D_n(m)$ implies that the function $D(z)$ can not be algebraic,
whereas for all even values of $m>2$, the asymptotics match the patterns appearing in the asymptotics of coefficients of algebraic functions. However, we expect the following conjectures to be true.
\begin{conj} For any rational function $p(z) \in \N(z)$  (with $p(1)>1$), the generating function $D(z)$ is not algebraic for $m>2$.
\end{conj}

It includes the specific case $D(z)= \sum_{n\geq 0}   D_n z^n$
with  $D_n= \sum_{k=0}^n \binom{n-k}{k}^m $ (non algebraicity of our initial puzzle)
or $D_n= \sum_{k=0}^n \binom{n}{k}^m   $ (non algebraicity of Franel numbers of order $m$). 
{\em Nota bene}: We gave here several ways to prove the non-algebraicity for some value of $m$, and we proved it for all odd $m>2$, we are however unaware of any way  of proving this for all even $m>2$ at once,
except,  in some cases, an evaluation at some  
$z$ leading to a transcendental number, 
or the Christol--Kamae--Mend\`es-France--Rauzy theorem on automatic sequences.

\begin{df}[Closed-form sequence] A sequence of integers $D_n$ is said to have a closed-form expression if it can be expressed as nested sums of hypergeometric terms, 
with natural boundaries  (i.e. the intervals of summation are 0 and $n$). N.b: the number of nested sums has to be independent of $n$.
\end{df}

Typical examples of closed-form expression 
are nested  sums of binomials; without loss (or win!) of generality, 
it is possible to allow more general intervals of summation or internal summands.

\begin{conj} Let D(z) be like in equation~\ref{eqnD} below (for any rational functions $p_i(z) \in \N(z)$\/),
then its coefficients $D_n$ can be expressed in closed form.
\end{conj}

 An effective way of finding this nested sum could be called a "reverse Zeilberger algorithm". It then  makes sense
to give the following broader conjecture: 
\begin{conj}
The coefficients of any D-finite function (i.e. a solution of a linear differential equation with polynomial coefficients)
can be expressed in closed form. 
\end{conj}
Note that it follows from the theory of G-series that this 
does not hold for closed-forms of the type "one sum of hypergeometric terms"~\cite{Ga09}.
The formulas we will give in the rest of this section are somehow 
illustrating these conjectures.

\subsection{Other nice explicit formulas}

It is clear from the previous subsection that we could play the same
game for any~$m$-tuple of compositions with parts restricted to~$m$
different sets, encoded by~$\p_1(z), \dots, \p_m(z)$. 
\begin{pr}\label{pr31} The generating function for the number of~$m$-tuples of 
compositions having the same number of parts is given by
\begin{equation}\label{eqnD}
D(z)=\frac{1}{(2i\pi)^{m-1}} \oint \frac{1}{1-\p_1(z t_2 \dots t_m)
 \p_2(1/t_2) \dots \p_m(1/t_m)} \frac{dt_2}{t_2} \dots
\frac{dt_m}{t_m}\,.
\end{equation}
\end{pr}

Therefore, one should not expect any nice closed-form solution for
$D(z)$ whenever~$m>2$; while for~$m=2$, whenever all the~$\p_i(z)$'s are
 polynomials or rational functions,~$D(z)$ will be an algebraic
function whose coefficients can be expressed by nested sums of binomial
coefficients (using Lagrange inversion). 

For example, if~$\p_1(z)=\p_2(z)=2z+z^2$ (which can be considered as tilings
with bicolored horizontal dominoes), 
one gets~$D_{n+1}(2)=\sum_{k=0}^n \sum_{j=0}^k \binom{k}{j}\binom{k+j}{j}$.

If~$\p_i(z)=\frac{z}{1-z}$ (i.e., we consider compositions with any
parts), then $D_n$ is the sequence of Franel numbers of order $m$:
$D_{n+1}(m)=\sum_{k=0}^n \binom{n}{k}^m$, and we will see in Section~\ref{sec5} that the probability 
that~$m$ unrestricted compositions of~$n$ have the same number of
parts is thus~$\pi_n \sim C_m/\sqrt{(\pi n)^{m-1}}~$,
with~$C_m=\sqrt{2^{m-1}/m}$.

Note that if we replace $p_i(z)$ (for $i>1$)  by $(1+p_i(z))$ in the integral formula of Proposition~\ref{pr31}, 
then this gives the generating function of tuples of compositions such that the {\em number} of parts is in decreasing order.

Let us add a few examples for which parts are in two different sets
$\mathcal P_1$ and~$\mathcal P_2$.
If~$\p_1(z)=z+z^2$ and~$\p_2(z)=z+2z^2$,
then one gets an interesting case as we have here 
$$\pi_n\sim
\frac{\sqrt{72+42\sqrt3} (\sqrt5-5) (\sqrt2-2)}{12 \sqrt{\pi n}}
\left(\frac{1-\sqrt2-\sqrt5+\sqrt{10}}{2 (2-\sqrt3)}\right)^n
\approx 1.62
\frac{.95^n}{\sqrt{\pi n}}\,,$$ which is therefore exponentially smaller
that the order of magnitude of our previous examples.
We will comment later on this fact.

Going to a slightly more general case~$\p_i(z)= \alpha_i z + \beta_i z^2$, one has for~$m=2$:
$$D(z)=\frac{1}
{\sqrt{1-2 \alpha_1 \alpha_2 z+ (\alpha_1^2 \alpha_2^2 -2 \beta_1 \beta_2) z^2-2 \alpha_1 \alpha_2 \beta_1 \beta_2 z^3 +\beta_1^2 \beta_2^2 z^4 }}\,.$$

Therefore the generating function only depends
on the products $\alpha_1 \alpha_2$ and $\beta_1 \beta_2$.
This implies e.g. that $p_1(z)= 2 z + 3 z^2$ and $p_2(z)= 3 z + 5 z^2$
will lead to the same $ D(z)$ as $p_1(z)=6 z + z^2$ and $p_2(z)= z + 15  z^2$.

Note that $D(z)$ factors nicely when~$\beta_1=\beta_2=1$:
$$D(z)=1/\sqrt{(\alpha_1 \alpha_2 z-1-2 z-z^2) (\alpha_1 \alpha_2 z-1+2 z-z^2)}\,.$$
If, additionally,~$\alpha_1 \alpha_2=\pm 4$ this gives the sum of
central Delannoy numbers~\cite{BaSc05}: 
$$D(z)=1/(1-z) \times 1/\sqrt{1+(2+|\alpha_1 \alpha_2|) z+z^2}\,.$$

 When the allowed parts are only $a$ and $b$, i.e. $P_i(z)=\alpha_i z^a + \beta_i z^b$,
then all the compositions in the constrained tuples have necessarily 
the same number of parts "$a$" (this also holds for the number of parts "$b$").
Choosing the order of the $n_1$ parts "$a$" and the $n_2$ parts "$b$"
leads to the formula $$D_n(m)=\sum_{ n_1 a+n_2 b =n} \binom{n_1+n_2}{n_1}^m (\alpha_1 \dots \alpha_m)^{n_1} (\beta_1 \dots \beta_m)^{n_2} .$$
There is no longer such a simple formula as soon as one has more than two allowed parts,
because the parts can then compensate each other in many ways,
e.g., assume that the allowed parts contain 3 integers $0<a<b<c$, 
then one can always create a composition $\mathcal{P}_1$ 
having $n_1$ "$a$", $n_2$ "$b$", $n_3$  "$c$" and a composition $\mathcal{P}_2$ having 
$m_1$ "$a"$, $m_2$ "$b$" and $m_3$ "$c$" such they have the same number of parts 
$n_1+n_2+n_3=m_1+m_2+m_3$, but   $(n_1,n_2,n_3) \neq (m_1,m_2,m_3)$.
To achieve this, consider 
$n_1= c-b$, $n_3= b-a$, $n_2=n_1+n_3$, $m_2=0$, $m_1=2 n_1$, $m_3=2 n_3$,
thus one gets two different compositions of $n$: $n=n_1 a+n_2 b+n_3 c =  m_1 a+ m_2 b + m_3 c$.

If~$\p_1(z)=\alpha z+ \beta z^2$ and~$\p_2(z)=z^2 /(1-z^2)$, then~$D_{2n}=\beta^n$, 
while if~$\p_1(z)=\alpha z+\beta z^2$ and~$\p_2(z)=z/(1-z)$, then
$$D(z)= \frac{1}{2}+\frac{1}{2} \frac{1+\alpha z}{\sqrt{1-2\alpha z+z^2 ( \alpha^2-4 \beta)}}\,.$$
So, a nice surprise is given by~$\p_1(z)=z+z^2$ and~$\p_2(z)=z/(1-z)$ ,
for which we get~$D(z)=1/2+1/2 \sqrt\frac{1+z}{1-3z}$, which is known to
be the generating function of directed animals \cite[A005773]{eis}.
This sequence also counts numerous other combinatorial structures:
variants of Dyck paths, pattern avoiding permutations, base 3 $n$-digit numbers with digit sum~$n$,\,\dots
 It also counts prefixes of Motzkin paths 
and this leads to an alternative formula $D_{n+1}=M_n=3^n -\sum_{k=0}^{n-1}  3^{n-k-1} E_{k} $, where $M_n$ and $E_n$ stand for meanders and excursions of length~$n$,
following the definitions and notations from~\cite{BaFl02}.

We leave
to the reader the pleasure of finding a bijective proof of all of 
this. (Some of them go via a bijection with lattice paths, as done in~\cite{bk},
and then via the bijection between heaps of pieces and directed animals, see Fig.~1.)
Note that some of the bijections can lead to efficient uniform random generation algorithms.

\begin{figure}[h]
\parbox{62mm}{
\includegraphics[height=25mm,width=3cm]{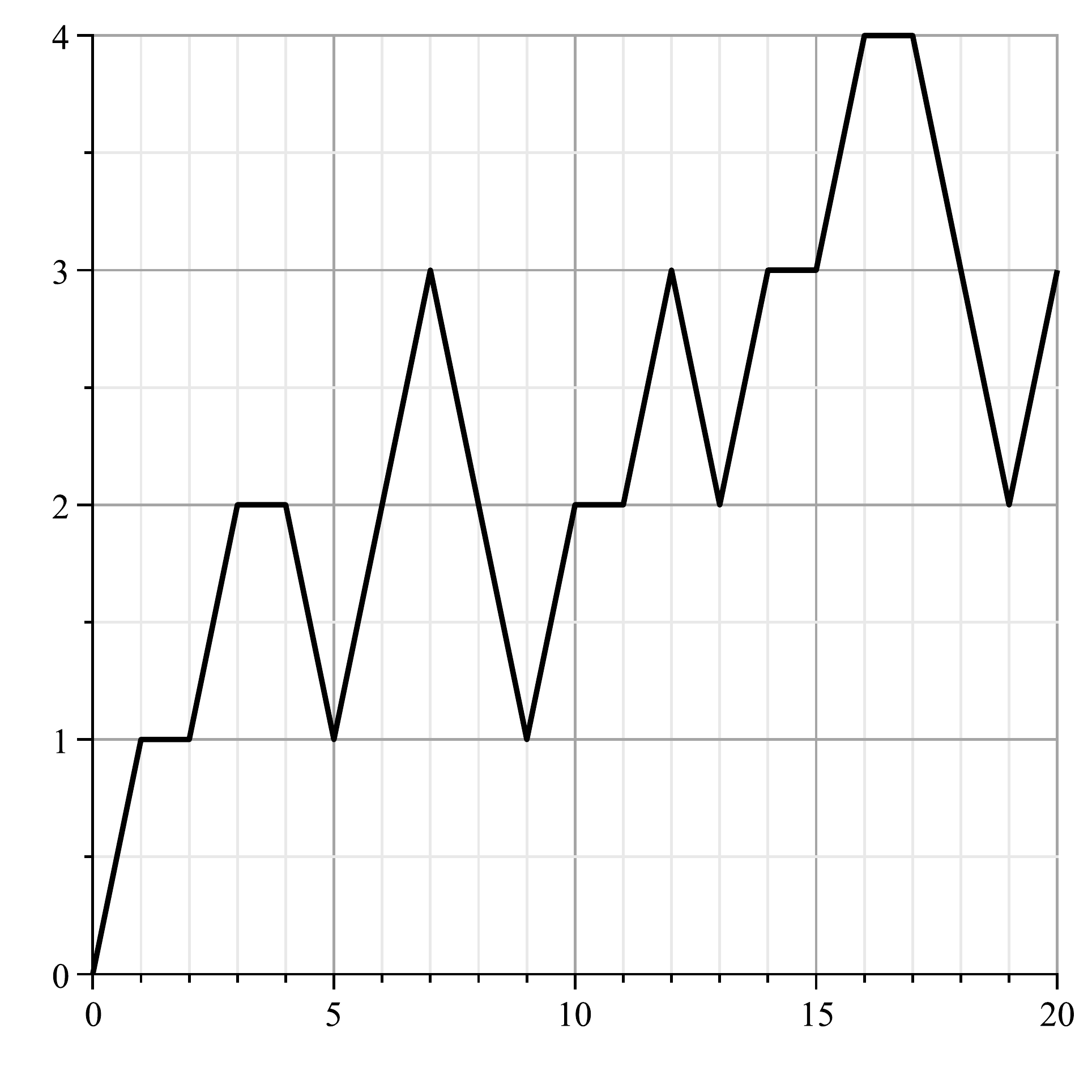}
\includegraphics[height=27mm]{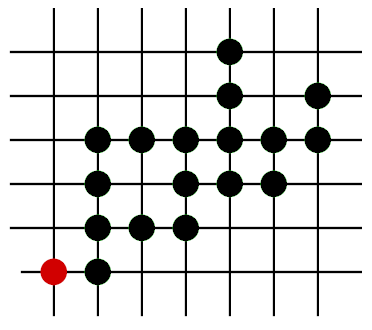}
}
\parbox{58mm}{
\caption{
Pairs of compositions having the same number of parts (e.g. (1,2,2,1,2,1,2,2,2,2,2) and (3,1,1,1,3,2,1,1,4,1,1))  are in bijection with several combinatorial objects, e.g. lattice paths (left) and directed animals (right).
}
}
\end{figure}

In summary, it may seem possible to compute everything in all cases;
however, for a generic~$\mathcal P$, in order to compute the constant~$C_m$ involved in
$\pi_n\sim C_m \frac{1}{\sqrt{(\pi n)^{m-1}}}$, we need heavier 
 computations if the degrees of the~$\p_i(z)$'s get large or if~$m$ is large.
Current state of the art algorithms will take more than one day for~$m=6$, and
gigabytes of memory, so this ``computer algebra'' approach (may it be
via guessing or via holonomy theory) has some intrinsic
limitations. What is more, for a given~$\mathcal P$, it remains a nice
challenge to get a rigorous (Zeilbergerian computer algebra) proof for
all~$m$ {\em at once}.

In the next sections, we show that the technical conditions to get a
local limit law hold, and that this allows to get
the constant~$C$, for any~$\mathcal P$, for all~$m$.

We conclude this section with Tables 1 and 2 (see next page) summarizing our main closed-form formulas.

\begin{table}
{\small \hspace{-4mm}
\setlength{\tabcolsep}{0cm}
\begin{tabular}{|c|c|c|}
\hline
allowed   & number of tuples  ($\mathcal P_1, \dots, \mathcal P_m$)  &  Sloane's On-line encyclopedia of integer sequences \\
parts in $\mathcal P_i$    & of compositions of $n$  having    & alternative description \\
& the same number of parts  & \\
\hline
&&The notation  $a^{\alpha_i}$ means that  the part $a$ is considered \\
$\{a^{\alpha_i},b^{\beta_i}\}$  & $D_n(m)=\displaystyle \sum_{k=0}^{n/b} \binom{\frac{n+k(a-b)}{a}}{k}^m (\alpha_1\dots\alpha_m)^{(n-kb)/a} (\beta_1\dots\beta_m)^k  $  
& with weight (or multiplicity) $\alpha_i$ in the compositions of  $\mathcal P_i$.  \\
&& (Binomials with fractional entries are considered as zero).\\
\hline
&&\\
$\{1,2\}$ &  $D_n(m)=\displaystyle \sum_{k=0}^n  \binom{n-k}{k}^m$  & $m=2$, A051286: Whitney number of level $n$ of the lattice  \\
&&of the ideals of the fence of order $2n$.\\
\hline
&&$m=2$, A089165: partial sums of the central Delannoy numbers, \\
$\{1,1,2\}$  &  $\displaystyle D_n(2)=\sum_{k=0}^{n-1}\sum_{j=0}^k \binom{k}{j}  \binom{k+j}{j}
=4^n \sum_{k=0}^n \binom{n-k}{k}^2/16^k$  &resistance between two nodes of an infinite lattice of unit resistors,\\
&& \# of peaks of odd level in Schroeder path.\\
\hline
&&Franel numbers \\
 $d \N$   &  $\displaystyle D_{dn}(m)=\sum_{k=0}^n  \binom{n}{k}^m$  &  ( $m=2$ simplifies to the 
central binomial numbers $\binom{2n}{n}$, A000984,  \\
&&  $m=3,4,5,6$: A000172, A005260 ,  A005261, A069865). \\
\hline
&&\\
$2 \N-1 $  & $\displaystyle D_{n+1}(m)= \sum_{k=0}^{n}  \binom{n-k}{k}^m$  & Same as  pairs of compositions of $n-1$ with parts in \{1,2\}  \\
&&  ($m=2,3,4,5$: A051286, A181545, A181546, A181547).\\
\hline
&& \\
$\{ n>1\}$  & $\displaystyle D_{n+2}(m)= \sum_ {k=0}^{n}  \binom{n-k}{k}^m$   & \\
&& \\
\hline
\end{tabular}
\caption{Summary of  the main closed-form formulas  (for any $m$) of our Section 3.}
}%end small
\end{table}

\begin{table}
{\small
\begin{tabular}{|c|c|c|c|c|}
\hline
allowed  &allowed &number of pairs ($\mathcal P_1, \mathcal P_2$) of &OEIS &alternative \\
parts &parts &compositions of $n$  having &reference 
 &OEIS description~\cite{eis} \\
 in $\mathcal P_1$ & in $\mathcal P_2$   &the same number of parts  && \\
\hline
&&&&\\
$\{1,2\}$ &$\{1,2,2\}$  &$\displaystyle \sum_{k=0}^n  \binom{n-k}{k}^2 2^k$  &A108488  &Expansion of $1/\sqrt{1-2x-3x^3-4x^3+4x^4}$.  \\
&&&&\\
\hline
&&&&\\
$\N  $ &$ \N \cup \{ 0 \} $  &  $\displaystyle  \sum_{k=0}^{n}  \binom{n-1}{k} \binom{n+k}{k}
=\sum_{k=0}^{n-1} \frac{n-k}{n}  \binom{n}{k}^2  2^{n-k-1}$
& A047781 & Convolution of central Delannoy numbers  \\
&&&& and Schroeder numbers. \\
\hline
&&&&Chebyshev transform of the central \\
$\{1,2\}$ & $\{1,1,2\}$  & $\displaystyle  \frac{1}{2^n}   \sum_{k=0}^n  (-1)^k   \binom{2k}{k}  \sum_{j=0}^{n-2k}  \binom{n-2k}{j}^2 3^k $
 &A101500&binomial numbers (the formulas\\
&&&& in this OEIS entry are not correct).\\
\hline
&&&&Some coefficients are 0, as gcd(parts) $\neq  1$.\\
$\N$ & $d\N$  &  $\displaystyle  D_{dn}= \binom{(d+1) n+d-1}{n}$  &A045721  &  For $d=2$, related to \\
&&&&lattice paths,  trees, standard tableaux...   \\
\hline
&&&A000079&\\
$\{1^\alpha,2^\beta\}$ & $2\N$  & $D_{2n}=\beta^n$  &  A000244  & The notation "$2^\beta$" means that the part $2$ \\
&&&A000302...&comes with a weight (or multiplicity) $\beta$.\\
\hline
&&&&\\
$\{1,2\}$ & $\N$  &  $\displaystyle  D_n=\sum_{k=0}^n \binom{n-1}{ k} \binom{n-k}{ k}$  &A005773  & directed animals (and numerous avatars  \\
&&&& of Motzkin paths, constrained matrices...).\\
\hline
\end{tabular}
\caption{Summary of  the main closed-form formulas (for $m=2$) of our Section 3.}
} %end small
\end{table}

\section{Local limit theorem for the number of parts in restricted compositions}\label{sec:local}
The discussion in this section pretty much gathers what has been
developed in various parts of the compendium on Analytic Combinatorics
by Flajolet \& Sedgewick~\cite{fs}. 

Our main generating function (see Equation~(\ref{compsch}))
is a particular case of a more general composition\footnote{We cannot
 escape this polysemy: Compositions are enumerated by a
 composition!} scheme 
considered in Flajolet and Sedgewick, namely~$F(z,u)=g(uh(z))$. In our
case~$g(y)=1/(1-y)$ and~$h(z)=\p(z)$. According to terminology used in
\cite[Definition~IX.2, p.~629, Sec. IX.3]{fs}, under our assumption
that~$\sum_{j\in\mathcal P}p_j>1$ the 
scheme is supercritical 
(i.e., when $z$ increases, one meets the singularity $y=1$ of $g(y)$ before any other potential singularity of $p(z)$). 
As a consequence, the number of parts~$X_n^{\mathcal P}$ is asymptotically normal as~$n\to\infty$, with both the mean and the variance linear in~$n$. We now briefly recapitulate the statements from~\cite{fs}. 
 The equation 
$\p(z)=1$ has a unique positive root~$\rho\in(0,1)$. As a consequence,~$F(z,1)$ has a dominant simple pole as its singularity and thus the number $P_n$ of compositions of~$n$ with all parts in~$\mathcal P$ is
\begin{equation}\label{asymptPn}
[z^n]F(z,1)\sim\frac1{\rho p'(\rho)}\rho^{-n}(1+O(\eps^n))\,,
\end{equation}
where~$\eps$ is a positive number less than 1, see~\cite[Theorem~V.1, p.~294]{fs}. 
The probability generating function of~$X_n^{\mathcal P}$ is given by
\[f_n(u)=\frac{[z^n]F(z,u)}{[z^n]F(z,1)}.
\]
In a sufficiently small neighborhood of~$u=1$, as a function of~$z$,~$F(z,u)$ given in (\ref{compsch}) has a dominant singularity~$\rho(u)$ which is the unique positive solution of the equation
 \[u \p(\rho(u))=1.\]
 Consequently, 
\[f_n(u)=\frac{[z^n]F(z,u)}{[z^n]F(z,1)}\sim \frac{p'(\rho(1))}{p'(\rho(u))}\cdot\left(\frac{\rho(u)}{\rho(1)}\right)^{-n-1}.
\]
It follows from the analysis of supercritical sequences given in~\cite[Proposition~IX.7, p. 652]{fs} that the number of parts~$X_n^{\mathcal P}$ satisfies
\[\frac{X_n^{\mathcal P}-EX_n^{\mathcal P}}{\sqrt{\mbox{var}(X_n^{\mathcal P})}}\stackrel d\longrightarrow N(0,1),\]
where~$N(0,1)$ denotes a standard normal random variable whose distribution function is given by
\[\Phi(x)=\frac1{\sqrt{2\pi}}\int_{-\infty}^xe^{-t^2/2}dt,\]
and where the symbol \lq\lq$\stackrel d\longrightarrow$\rq\rq\ denotes the convergence in distribution. The asymptotic expressions for the expected value and the variance of~$X_n^{\mathcal P}$ are given by

\begin{equation}\label{eqn:mu} 
EX_n^{\mathcal P}=\frac n{\rho p'(\rho)}+O(1)   \text{\qquad with $\rho\in(0,1)$ such that $p(\rho)=1$,}   
\end{equation}
\begin{equation}\label{eqn:var} 
\mbox{var}(X_n^{\mathcal P})=K n+O(1) \text{\qquad where } K=\frac{\rho p''(\rho)+p'(\rho)-\rho(p'(\rho))^2}{\rho^2(p'(\rho))^3}\,.  \end{equation}

[Note that the expression for the coefficient $K$ of the variance given in Proposition~IX.7 in~\cite{fs} is incorrect; the correct version is~$(\rho h''(\rho)+h'(\rho)-\rho h'(\rho)^2)/(\rho h'(\rho)^3)$, as given in many other places in the book.]

We now note that the central limit theorem can actually be
strengthened to the local limit theorem, pretty much as discussed
in~\cite[Theorem~IX.14 and the remarks following its proof on
p.~697]{fs}. Let us recall the following notion.
\begin{df}\label{def:llc} Let~$(X_n)$ be a sequence of integer valued random variables with $EX_n=\mu_n$ and~${\rm var}(X_n)=\sigma_n^2$.
Let~$(\eps_n)$ be a sequence of positive numbers going to 0. We say that~$(X_n)$ satisfies a local limit theorem (of Gaussian type) with speed~$\eps_n$ if 
\[\sup_{x\in {\mathbb R}}\left|\sigma_n\operatorname{Pr}(X_n=\lfloor \mu_n+x\sigma_n\rfloor)-\frac{e^{-x^2/2}}{\sqrt{2\pi}}\right|\le\eps_n.\]
\end{df}
As was discussed in~\cite[p.~697]{fs}, to see that the local limit
theorem holds for restricted compositions, it suffices to check that
$\rho(u)$ when restricted to the unit circle uniquely attains its
minimum\footnote{There is a typo in~\cite{fs} p.697: the
 inequality direction is wrong.} at~$u=1$.
This is what we prove in the following lemma.
\begin{lem} 
Let~$p$ be a power series with nonnegative coefficients, of radius of
convergence~$\rho_\p>0$ (possibly~$\rho_\p=+\infty$).
Let~$\rho(u)$ be as above the positive root\footnote{$\p(z)$ has
 nonnegative real coefficients and is thus increasing in a
neighborhood of 0, i.e. on~$z\in[0,+\epsilon]$.~$p$ being analytic near 0, is continuous and for any~$x \in {\mathbb R}$ small enough,
$\p(z)=x$ will therefore have a real positive root~$z_x$, and this root will be analytic in~$x$.
This is the root that we call ``real positive''.} 
of~$u \p(\rho(u))=1$. If~$\p$ is aperiodic\footnote{A power series~$\p$ 
is said to be periodic if and only if there exists a power series~$q$ and an integer~$g>1$ such that~$\p(z)=q(z^g)$.
Equivalently, the gcd of the support (=the ranks of nonzero
coefficients) of the power series~$\p$ is~$g\neq 1$. If this gcd~$g$
equals 1,
then~$p$ is said to be aperiodic.}, 
then for~$0<\frac{1}{R}<\rho_\p$ and~$t \in]0,2\pi[$, we have 
\[\rho(R) < |\rho(R e^{it})| \,,\]
i.e. the minimum on each circle is on the positive real axis.
In particular, if the radius of convergence of~$\p$ is larger than 1,
then for~$|u|=1$ and~$u\ne1$ we have \[ \rho(1)< |\rho(u)|.\]
\end{lem} 
\begin{proof}
First,~$\p$ has nonnegative real coefficients, therefore
the triangle inequality gives~$\p(|\rho(u)|) \geq |\p(\rho(u))|$.
Equality can hold only if~$\p(\rho(u))$ has just nonnegative terms,
but this is not possible if~$\rho(u)\not \in {\mathbb R}^+$ as~$\p$ is
aperiodic with nonnegative coefficients.
Hence one has a strict triangle inequality:~$\p(|\rho(u)|) > |\p(\rho(u))| = |1/u| =1/R$ (the
middle equality is just the definition of~$\rho$ and the last
equality comes from the fact we are on the circle~$|u|=R$).
As~$\p$ is increasing on~$[0,1/R]$, we can apply~$\p^{-1}$ 
to~$\p(|\rho(u)|) > 1/R$ 
which gives~$\p^{-1}(\p(|\rho(u)|)> \p^{-1}(1/R)$, that is~$|\rho(u)|> \rho(R)$.
\end{proof}

Note that the aperiodicity condition is important,
e.g. for~$\p(z)=z^2+z^6$ (i.e.~$\rho(u)$ is the radius of convergence
of~$P(z,u)=1/(1-u (z^2+z^6))$), one has~$\rho(-1) = i \rho(1)$; 
however some periodic cases have a unique minimum on the unit circle, e.g.~$\p(z)=z^2+z^4$.
Note also that (in either periodic or aperiodic case),
the uniqueness of the minimum on the circle~$|u|=1$ at~$u=1$ 
does not hold in general for the other roots of~$u \p(\rho(u))=1$.

\section{Asymptotic probability that restricted compositions have the same number of parts}\label{sec5}
Our motivation for including the results from~\cite{fs} in the preceding section is the following theorem which considerably extends the main results of~\cite{bk}. We single out the case~$m=2$ since in some cases it has been already studied in the literature.

\subsection{Pairs of compositions}
\begin{thm}\label{thm:samenoparts} Let~$\mathcal P\subset\N$. 
The probability that two random compositions with parts in~$\mathcal P$ have the
 same number of parts is, asymptotically as~$n\to\infty$,
\[\pi_n\sim\frac C{\sqrt\pi\sqrt n},\]
where the value of~$C$ is related to the constant $K$ from
Equation~\ref{eqn:var}, namely:  
\begin{equation}\label{eqn:c}C=\frac{1}{2} \sqrt{K}= \frac{\rho(p'(\rho))^{3/2}}{2\sqrt{\rho p''(\rho)+p'(\rho)-\rho(p'(\rho))^2}}.
\end{equation}
\end{thm}
Before proving this theorem let us make some comments.
 
\noindent{\bf Remarks and examples:} 
\begin{itemize}
\item[(i)]\label{rem:(i)} Some special cases were considered in~\cite{bk}. They include unrestricted compositions ($\mathcal P=\N$),~$\mathcal P=\{1,2\}$, or more generally~$\mathcal P=\{a,b\}$ with~$a$,~$b$ relatively prime,
compositions with all parts of size at least~$d$ ($\mathcal P=\{n\in\N:\, n\ge d\}$), and compositions with all parts odd and at least~$d$. The arguments of~\cite{bk} rely on the analysis of the asymptotics of the bivariate generating functions, which is sometimes difficult and does not seem to be easily amenable to the analysis in the case of a general subset~$\mathcal P$ of positive integers. Our approach is much more probabilistic and relies on a local limit theorem for the number of parts in a random composition with parts in~${\mathcal P}$. This turned out to be a much more universal tool. 
\item[(ii)] To illustrate the principle behind our approach, consider the unrestricted compositions. As was observed in~\cite{hs}, in that case~$X_n^{\mathcal P}$ is distributed like~$1+{\rm Bin}(n-1,1/2)$ random variable. 
Therefore, 
\[\pi_n=\operatorname{Pr}({\rm Bin}(n-1,1/2)={\rm Bin'}(n-1,1/2))\]
where~${\rm Bin}$ and~${\rm Bin'}$ denote two independent binomial random variables with specified parameters. 
Since the second parameter is~$1/2$ we have \[{\rm Bin}(n-1,1/2)\stackrel d=n-1-{\rm Bin}(n-1,1/2).\] Therefore, by independence we get
\[
\pi_n=\operatorname{Pr}({\rm Bin}(n-1,1/2)+{\rm Bin'}(n-1,1/2)=n-1).\]
Finally, since 
\[{\rm Bin}(n-1,1/2)+{\rm Bin'}(n-1,1/2)\stackrel d={\rm Bin}(2(n-1),1/2),\]
we obtain by Stirling's formula that
\[\pi_n=\operatorname{Pr}({\rm Bin}(2n-2,1/2)=n-1)=\frac{{2n-2\choose n-1}}{2^{2n-2}}\sim\frac1{\sqrt{\pi n}}\,.\]
This is consistent with (\ref{eqn:c}) (and with
\cite{bk}) as for unrestricted compositions
$\p(z)=\sum_{k\ge1}z^k=z/(1-z)$, so that~$\rho=1/2$,~$\p'(z)=1/(1-z)^2$,
and~$\p''(z)=2/(1-z)^3$ which gives~$C=1$.

\item[(iii)] Although the above argument may look very special and heavily reliant on the properties of binomial random variables, our point here is that it is actually quite general. The key feature is that the number of parts (whether in unrestricted or arbitrarily restricted compositions) satisfies the local limit theorem of Gaussian type, and this is enough to asymptotically evaluate the probability in Theorem~\ref{thm:samenoparts}.

\item[(iv)] For another example, consider compositions of~$n$ into two parts, i.e.~$\mathcal P=\{a,b\}$ with~$a$,~$b$ relatively prime. Then Theorem~\ref{thm:samenoparts} holds with 
\begin{equation}\label{eqn:cab}C=\frac{(a\rho^a+b\rho^b)^{3/2}}{2|a-b|\sqrt{\rho^{a+b}}},\end{equation}
where~$\rho$ is the unique root of~$z^a+z^b=1$ in the interval~$(0,1)$.
In this case~$\p(z)=z^a+z^b$ so that~$p'(z)=az^{a-1}+bz^{b-1}$ and~$p''(z)= a(a-1)z^{a-2}+b(b-1)z^{b-2}$. Thus, writing the numerator of \eqref{eqn:c} as 
\[\rho(P'(\rho))^{3/2}=\frac1{\sqrt\rho}(\rho P'(\rho))^{3/2}=\frac1{\sqrt\rho}(a\rho^a+b\rho^b)^{3/2},\]
we only need to check that
\begin{equation}\label{eqn:rho}\rho P''(\rho)+P'(\rho)-\rho(P'(\rho))^2=(a-b)^2\rho^{a+b-1}.\end{equation}
But
\[\rho P''(\rho)+P'(\rho)=a^2\rho^{a-1}+b^2\rho^{b-1}\]
so that the left--hand side of \eqref{eqn:rho} is 
\[a^2\rho^{a-1}+b^2\rho^{b-1}-a^2\rho^{2a-1}-b^2\rho^{2b-1}-2ab\rho^{a+b-1}.\]
Factoring and using~$\rho^a+\rho^b=1$, we see that this is 
\[a^2\rho^{a-1}(1-\rho^a)+b^2\rho^{b-1}(1-\rho^b)-2ab\rho^{a+b-1}=(a-b)^2\rho^{a+b-1},
\]
as claimed.

When~$a=1$ and~$b=2$ we have the Fibonacci numbers relation so that~$\rho=(\sqrt5-1)/2$ and \eqref{eqn:cab} becomes
\[C=\frac{(\rho+2\rho^2)^{3/2}}{2\rho^{3/2}}=\frac12(1+2\rho)^{3/2}=\frac{5^{3/4}}2,\]
which agrees with \eqref{eqn:pn} above and also with the expression given in~\cite{bk}
(see equation (2.10) therein). However, in the case of general~$a$ and~$b$, the value of~$C$ was given in the last display of Section~3 in~\cite{bk} as
\begin{equation}\label{eqn:cabbk}\frac{\rho(a\rho^{a-1}+b\rho^{b-1})^2}{\sqrt{4(a+b)\rho^{2a+2b-2}+2(1-\rho^{2a}-\rho^{2b})(a\rho^{2a-2}+b\rho^{2b-2})}}.\end{equation}
This is incorrect as it is lacking a factor~$|a-b|$ in the denominator (so that it gives the correct value of~$C$ when~$|a-b|=1$ but not otherwise). To see this and also to reconcile \eqref{eqn:cabbk} with \eqref{eqn:cab} (up to a factor~$|a-b|$) we simplify \eqref{eqn:cabbk} by noting that~$\rho^a+\rho^b=1$ implies 
\[1-\rho^{2a}-\rho^{2b}=1-(\rho^a)^2-\rho^{2b}=(1+\rho^a)\rho^b-\rho^{2b}=\rho^b(1+\rho^a-\rho^b)=2\rho^{a+b}
\]
so that the expression under the square root sign in \eqref{eqn:cabbk} becomes
\[4\rho^{a+b-2}((a+b)\rho^{a+b}+a\rho^{2a}+b\rho^{2b})= 
4\rho^{a+b-2}(a\rho^{a}+b\rho^{b})(\rho^a+\rho^b).
\] 
 Using again~$\rho^a+\rho^b=1$ \eqref{eqn:cabbk} is seen to be 
 \[\frac{\rho^2(a\rho^{a-1}+b\rho^{b-1})^2}{2\sqrt{\rho^{a+b}(a\rho^a+b\rho^b)}}=\frac{(a\rho^a+b\rho^b)^{3/2}}{2\sqrt{\rho^{a+b}}},
 \]
 which, except for the factor~$|a-b|$ in the denominator, agrees with \eqref{eqn:cab}.

\item[(v)] Other examples from~\cite{bk} can be rederived in the same
 fashion, but we once again would like to stress universality of our
 approach. As an extreme example, we can only repeat after~\cite{fs}:
 even if we consider compositions into twin primes,~$\mathcal
 P=\{3,5,7,11,13,17, 19, 29,31,\dots\}$, we know that the
 probability of two such compositions having the same number of parts
 is of order~$1/\sqrt n$. This is rather remarkable, considering the
 fact that it is not even known whether this set~$\mathcal P$ is
 finite or not. 
\end{itemize}

\noindent{\bf Proof of Theorem~\ref{thm:samenoparts}.} This will follow immediately from the following lemma applied to~$X_n=X_n^{\mathcal P}$ and formula \eqref{eqn:var} which gives the expression for 
$\sigma_n$. This lemma should be compared with a more general Lemma~6 of~\cite{bf}. We include our proof to illustrate that seemingly very special arguments used in item~(ii) are actually quite general.\qed
\begin{lem}\label{lem:llc0}
Let~$(X_n)$ with~$EX_n=\mu_n$ and~${\rm var}(X_n)=\sigma_n^2\to\infty$ as~$n\to\infty$, be a sequence of integer valued random variables satisfying a local limit theorem (of Gaussian type) with speed~$\eps_n$ as described in Definition~\ref{def:llc}. Let~$(X_n')$ be an independent copy of~$(X_n)$ defined on the same probability space. Then 
\[\pi_n=\operatorname{Pr}(X_n=X_n')=\frac1{2\sqrt{\pi}\sigma_n}+O\left(\frac{\eps_n}{\sigma_n}+\frac1{\sigma_n^2}\right).\]
\end{lem}

\begin{proof}
For~$X_n$ and~$X_n'$ as in the statement 
 we have
\begin{eqnarray}\nonumber
\pi_n=\operatorname{Pr}(X_n=X_n')&=&\sum_{k\ge1}\operatorname{Pr}(X_n=k=X_n')=\sum_{k\ge1}\operatorname{Pr}^2(X_n=k)\\
&=&\sum_{\ell=-\infty}^\infty\operatorname{Pr}(X_n=\lfloor\mu_n\rfloor+\ell)\operatorname{Pr}(X_n=\lfloor\mu_n\rfloor+\ell).\label{eqn:sumprod}
\end{eqnarray}
Now,
\[\operatorname{Pr}(X_n=\lfloor\mu_n\rfloor+\ell)=\operatorname{Pr}(X_n=\lfloor\mu_n\rfloor-\ell)+\Big\{\operatorname{Pr}(X_n=\lfloor\mu_n\rfloor+\ell)-
\operatorname{Pr}(X_n=\lfloor\mu_n\rfloor-\ell)\Big\}.
\]
To estimate the term in the curly brackets take~$x_+$ and~$x_-$ such that
\[\lfloor \mu_n\rfloor+\ell=\lfloor\mu_n+x_+\sigma_n\rfloor,\quad\mbox{and}\quad
\lfloor \mu_n\rfloor-\ell=\lfloor\mu_n-x_-\sigma_n\rfloor.
\]
By elementary considerations,
$-2 \{\mu_n\}/\sigma_n \le x_{+}-x_{-}\le 2(1-\{\mu_n\})/\sigma_n$
(where~$\{z\}$ is the fractional part of~$z$),
hence~$|x_+-x_-|\le2/\sigma_n$. Then
$$\operatorname{Pr}(X_n=\lfloor\mu_n\rfloor+\ell)-\operatorname{Pr}(X_n=\lfloor\mu_n\rfloor-\ell)  = \frac1{\sqrt{2\pi}\sigma_n}\left(e^{-\frac{x_+^2}2}-e^{-\frac{x_-^2}2}\right) \qquad \qquad$$
$$+\left(\operatorname{Pr}(X_n=\lfloor\mu_n+x_+\sigma_n\rfloor)-\frac{e^{-x_+^2/2}}{\sqrt{2\pi}\sigma_n}\right)
-\left(\operatorname{Pr}(X_n=\lfloor\mu_n-x_-\sigma_n\rfloor)-\frac{e^{-x_-^2/2}}{\sqrt{2\pi}\sigma_n}\right).$$
The absolute value of the second term is 
\[\frac1{\sigma_n}\left|\sigma_n\operatorname{Pr}(X_n=\lfloor\mu_n+x_+\sigma_n\rfloor)-\frac1{\sqrt{2\pi}}e^{-\frac{x_+^2}2}\right|\le\frac{\eps_n}{\sigma_n},
\]
\noindent and similarly with the third term. 
Applying the inequality~$|f(x)-f(y)|\le |x-y| \sup |f'(t)|$ to the
first term gives
$e^{-\frac{x_+^2}2}-e^{-\frac{x_-^2}2} \le |x_+-x_-| = O(1/\sigma_n)$ and so, the first term is~$O(1/\sigma^2_n)$. Therefore,
\[\Big|\operatorname{Pr}(X_n=\lfloor\mu_n\rfloor+\ell)-
\operatorname{Pr}(X_n=\lfloor\mu_n\rfloor-\ell)\Big|=O\left(\frac{\eps_n}{\sigma_n}+\frac1{\sigma_n^2}\right).
\]
Coming back to equation (\ref{eqn:sumprod}), we see that 
\begin{eqnarray*}&&\sum_{\ell=-\infty}^\infty \operatorname{Pr}(X_n=\lfloor\mu_n\rfloor+\ell)\left(\operatorname{Pr}(X_n=\lfloor\mu_n\rfloor-\ell)+O\left(\frac{\eps_n}{\sigma_n}+\frac1{\sigma_n^2}\right)\right)
\\ &&\quad=
\left(\sum_{\ell=-\infty}^\infty\operatorname{Pr}(X_n=\lfloor\mu_n\rfloor+\ell,
X_n'=\lfloor\mu_n\rfloor-\ell)\right)+1 \times
O\left(\frac{\eps_n}{\sigma_n}+\frac1{\sigma_n^2}\right)\\ &&\quad=
\operatorname{Pr}(X_n+X_n'=2\lfloor\mu_n\rfloor)
+O\left(\frac{\eps_n}{\sigma_n}+\frac1{\sigma_n^2}\right).
\end{eqnarray*}
Since~$X_n+X_n'$ is  a sum of two i.i.d. random variables, it has mean~$2\mu_n$ and the variance~$2\sigma_n^2$. Furthermore, since each of the summands satisfies the local limit theorem of Gaussian type, so does the sum (its probability generating function is the square of~$f_n(u)$ and thus falls into quasi-power category, just as~$f_n(u)$ does). Since~$2\lfloor\mu_n\rfloor=\lfloor2\mu_n+x\sqrt2\sigma_n\rfloor$ for some~$x=O(1/\sigma_n)$, just as before we have
\[\left|\sqrt2\sigma_n\operatorname{Pr}(X_n+X_n'=\lfloor2\mu_n\rfloor)-\frac1{\sqrt{2\pi}}\right|=O\left(\eps_n+\frac1{\sigma_n}\right).\]
Consequently,
\[\pi_n=\operatorname{Pr}(X_n=X_n')=\operatorname{Pr}(X_n+X_n'=\lfloor2\mu_n\rfloor)=\frac1{2\sqrt{\pi}\sigma_n}+O\left(\frac{\eps_n}{\sigma_n}+\frac1{\sigma_n^2}\right),\]
which completes the proof of Lemma~\ref{lem:llc0} and of Theorem~\ref{thm:samenoparts}.
\end{proof}
\subsection{Tuples of compositions}
Here we sketch a proof of the following extension of Theorem~\ref{thm:samenoparts}. 
\begin{thm}\label{thm:msamenoparts} Let~$\mathcal P\subset\N$ and let
~$m\ge2$ be fixed. 
Then, the probability~$\pi_{n}$ that~$m$ randomly and independently chosen compositions with parts in~$\mathcal P$ all have the
 same number of parts is, asymptotically as~$n\to\infty$,
\[\pi_{n}\sim\frac {C_m}{\sqrt{(\pi n)^{m-1}}},\]
where $C_m$ is related to the constant $K$ from
Equation~\ref{eqn:var}, namely: 
\[C_m=\frac1{\sqrt{2^{m-1}m}} \sqrt{K}^{m-1}= \frac1{\sqrt{2^{m-1}m}}\left(\frac{\rho^2(p'(\rho))^{3}}{\rho p''(\rho)+p'(\rho)-\rho(p'(\rho))^2}\right)^{(m-1)/2}.
\]
\end{thm}
{\bf Remark.} For unrestricted compositions, the expression in the big parentheses is 2 (see (i) in Remarks above). This gives~$C_m=\sqrt{2^{m-1}/m}$ as stated in Section~\ref{sec:dom}.

\noindent{\it Proof of Theorem~\ref{thm:msamenoparts}.} This follows immediately from the following statement which itself is a straightforward extension of Lemma~6 of~\cite{bf} with essentially the same proof. We will be using it for Gaussian density in which case
\[\int_{-\infty}^{\infty}g^m(x)dx=\frac1{\sqrt{(2\pi)^m}}\int_{-\infty}^\infty e^{-\frac{mx^2}2}dx=\frac1{\sqrt{(2\pi)^{m-1}}}\ \frac1{\sqrt m}.
\]
 
\begin{lem}\label{lem:flaj} {\rm (B\'ona--Flajolet).} Let~$(X_n)$ be integer valued with~$\mu_n=EX_n$,~$\sigma_n^2=\operatorname{var}(X_n)\to\infty$ as~$n\to\infty$. Let~$g$ be the probability density function and suppose that
\[\lim_{n\to\infty}\sup_{x}|\operatorname{Pr}(X_n=\lfloor\mu_n+x\sigma_n\rfloor)-g(x)|=0.\]
Let further~$(X_n^{(k)})$,~$k=1,\dots,m$ be independent copies of the sequence~$(X_n)$ defined on the same probability space. Then 
\[\sigma_n^{m-1}\operatorname{Pr}(X_n^{(1)}=X_n^{(2)}=\dots=X_n^{(m)})\longrightarrow\int_{-\infty}^\infty g^m(x)dx,\quad\mbox{as}\quad n\to\infty.\]
\end{lem}
To see this we just follow the argument in~\cite[Lemma~6]{bf} with obvious adjustments: the left--hand side above is
\[\sigma_n^{m-1}\sum_{k=1}^\infty\operatorname{Pr}^m(X_n=k)=\sigma_n^{m-1}\sum_{k=1}^\infty\operatorname{Pr}^m(X_n=\lfloor\mu_n+x_k\sigma_n\rfloor),
\]
with~$\frac{k-\mu_n}{\sigma_n}\le x_k<\frac{k+1-\mu_n}{\sigma_n}$. This is further equal to 
\[\frac1{\sigma_n}\sum_k(\sigma_n\operatorname{Pr}(X_n=\lfloor\mu_n+x_k\sigma_n\rfloor))^m\sim\frac1{\sigma_n}\sum_kg^m(x_k)\sim\frac1{\sigma_n}\int_{-\infty}^\infty g^m(\frac{x-\mu_n}{\sigma_n})dx,\]
where the first approximation holds by the assumption of the lemma (after having first restricted the range of~$x_k$'s to a large compact set) and the second by the Riemann sum approximation of the integral. Since the expression on the right is~$\int_{-\infty}^\infty g^m(x)dx$, the result follows.\hfill$\square$

\section{Concluding remarks}

\begin{enumerate}
\item
In this article, we restricted our attention to compositions (giving
first several new closed-form formulas, and then going to the asymptotics), 
but it is clear that Lemma~\ref{lem:flaj} can be applied to many
 combinatorial structures, e.g. the probability that~$m$ random
permutations of size~$n$ have the same number 
of cycles (see~\cite{wilf} for the case~$m=2$), or the probability that $m$ permutations 
have a longest increasing subsequence of the same length,
or the probability that~$m$ random planar maps have a largest component
of same size. This leads to interesting analytic/computational considerations,
as it will involve evaluating the integral of~$g^m(x)$ where~$g(x)$
will be the Tracy--Widom distribution density (provided the local limit
theorem holds, which has not been proven yet), or the map-Airy
distribution density (for which a local limit theorem was established,
see~\cite{BaFlScSo01}).

\item A similar approach can be also applied to tuples of combinatorial
structures following~$m$ different local limit laws (with~$m$ densities
having fast decreasing tails), as long as they have the same mean.

\item 
When the means are not the same, the probability of the same number of parts is generally of much smaller order. This is because if~$X_n$ has mean~$cn$ and~$X'_n$ has mean~$c'n$ and both have linear variances, then assuming w.l.o.g.~$c>c'$ and choosing~$\alpha<\frac{c-c'}2$ we note that if~$|X_n-cn|< \alpha n$ and~$|X'_n-c'n|<\alpha n$ then 
\[X_n-X'_n>cn-\alpha n-(c'n+\alpha n)=(c-c'-2\alpha)n>0,\]
so that~$X_n\ne X'_n$. Therefore,
\[\pi_n=\operatorname{Pr}(X_n=X'_n)\le\operatorname{Pr}(|X_n-cn|\ge\alpha n)+\operatorname{Pr}(|X'_n-c'n|\ge\alpha n).\]
Since both~$X_n$ and~$X'_n$ converge to a Gaussian law and~$\sigma_n=\sigma n$, the first probability is roughly (with~$\beta=\alpha/\sqrt\sigma$)
\[\operatorname{Pr}\left(\frac{|X_n-cn|}{\sqrt{\sigma n}}\ge\beta \sqrt n\right)\sim\frac1{\sqrt{2\pi}}\int_{\beta\sqrt n}^\infty e^{-t^2/2}dt\sim\frac1{\sqrt{2\pi}\beta \sqrt n}e^{-\frac{\beta^2n}2},
\]
by the well--known bound on the tails of Gaussian random variables (see, e.g.~\cite[Chapter~VII, Lemma~2]{fel}). This is consistent with an example discussed in Section~\ref{sec:dom}. The difficulty with making this argument rigorous is that the error in the first approximation is usually of much bigger (typically~$1/\sqrt n$) magnitude than the quantities that are approximated. However, a slightly weaker bound, namely,~$e^{-\beta n}$ (with a generally different value of~$\beta$) can be obtained by using Theorem~IX.15 in~\cite{fs} which asserts that tail probabilities of random variables falling in the scheme of quasi-powers are decaying exponentially fast. While this theorem is stated for the logarithm of~$\operatorname{Pr}(|X_n-cn|>\alpha n)$, it is clear from its proof that one actually gets exponential bound on the tail probabilities (see Equation~(88) on p.~701 in~\cite{fs} and a few sentences following it). 

\item The Gaussian local limit law explains the universality of the
$1/(\pi n)^{(m-1)/2}$ appearance for numerous combinatorial problems in which
we would force~$m$ combinatorial structures of size~$n$ to have 
an extra parameter of the same value. We also wish to point out yet another insight provided by the probabilistic approach. As we mentioned in Section~\ref{sec:dom} (see Footnote~\ref{foot:frob}), it 
allows to solve the connection constant problem intrinsic to the Frobenius
method, and therefore, a combination of these two approaches (local limit
law plus Frobenius method) gives access to full asymptotics in
numerous cases.
\end{enumerate}

\bigskip

{\bf Acknowledgements.} Part of this work was done during 
Pawe{\l}  Hitczenko's sojourn at LIPN (Laboratoire d'Informatique de
Paris Nord), thanks to the invited professor position
funded by the university of Paris 13 and the Institute
Galil\'ee. He would like to thank the members of LIPN for their hospitality. 

{\small
%\bibliographystyle{plain}
%\bibliography{BaHi}

\begin{thebibliography}{10}

\bibitem{and}
George~E. Andrews.
\newblock {\em The theory of partitions}.
\newblock Addison-Wesley, 1976.
\newblock Encyclopedia of Mathematics and its Applications, Vol. 2.

\bibitem{BaFl02}
Cyril Banderier and Philippe Flajolet.
\newblock Basic analytic combinatorics of directed lattice paths.
\newblock {\em Theoretical Computer Science}, 281(1-2):37--80, 2002.

\bibitem{BaFlScSo01}
Cyril Banderier, Philippe Flajolet, Gilles Schaeffer, and Mich\`ele Soria.
\newblock {Random maps, coalescing saddles, singularity analysis, and Airy
  phenomena.}
\newblock {\em Random Struct. Algorithms}, 19(3-4):194--246, 2001.

\bibitem{BaSc05}
Cyril Banderier and Sylviane Schwer.
\newblock {Why Delannoy numbers?}
\newblock {\em J. Stat. Plann. Inference}, 135(1):40--54, 2005.

\bibitem{bf}
Mikl\'os B\'ona and Philippe Flajolet.
\newblock Isomorphism and symmetries in random phylogenetic trees.
\newblock {\em J. Appl. Probab.}, 46(4):1005--1019, 2009.

\bibitem{bk}
Mikl\'os B\'ona and Arnold Knopfmacher.
\newblock On the probability that certain compositions have the same number of
  parts.
\newblock {\em Ann. Combinatorics}, 10:291--306, 2010.

\bibitem{BostanSchost}
Alin Bostan and {\'E}ric Schost.
\newblock Fast algorithms for differential equations in positive
  characteristic.
\newblock In John May, editor, {\em ISSAC'09}, pages 47--54, 2009.

\bibitem{fel}
William Feller.
\newblock {\em An introduction to probability theory and its applications.
  {V}ol. {I}}.
\newblock Third edition. John Wiley \& Sons Inc., New York, 1968.

\bibitem{fs}
Philippe Flajolet and Robert Sedgewick.
\newblock {\em Analytic combinatorics}.
\newblock Cambridge University Press, Cambridge, 2009.

\bibitem{Ga09}
Stavros Garoufalidis.
\newblock {G-functions and multisum versus holonomic sequences.}
\newblock {\em Advances in Mathematics}, 220(6):1945--1955, 2009.

\bibitem{hs}
Pawe{\l} Hitczenko and Carla~D. Savage.
\newblock On the multiplicity of parts in a random composition of a large
  integer.
\newblock {\em SIAM J. Discrete Math.}, 18(2):418--435, 2004.

\bibitem{eis}
Neil J.~A. Sloane.
\newblock The {O}n-{L}ine {E}ncyclopedia of {I}nteger {S}equences.
\newblock Published electronically at
  \texttt{www.research.att.com/$\sim$njas/sequences/}, 2011.

\bibitem{wilf}
Herbert~S. Wilf.
\newblock The variance of the {S}tirling cycle numbers.
\newblock Available at \texttt{http://arxiv.org/abs/math/0511428v2}, 2005.

\bibitem{WimpZeilberger85}
Jet Wimp and Doron Zeilberger.
\newblock {Resurrecting the asymptotics of linear recurrences.}
\newblock {\em J. Math. Anal. Appl.}, 111:162--176, 1985.

\end{thebibliography}
%} 

}

\end{document}